\documentclass[11pt]{article}
\RequirePackage[
 left=20mm,
 top=20mm, bottom=20mm]{geometry}
\RequirePackage{cite}
\RequirePackage{hyperref}
\RequirePackage{amsthm,amssymb}
\usepackage{float}
\usepackage[T1]{fontenc} %
\usepackage[table]{xcolor}
\usepackage{tikz}
\usepackage{tikz-cd}
\usepackage{forloop}
\usepackage{pgf} 
\usepackage{graphicx}
\usepackage{amsmath}
\usepackage{setspace}
\usepackage{multirow}
\usepackage{adjustbox}
\usepackage{authblk}
\usepackage[utf8]{inputenc}

\usetikzlibrary{decorations.pathmorphing}
\usetikzlibrary{positioning}
\usetikzlibrary{calc}
\usetikzlibrary{shapes.misc}
\usetikzlibrary{math}

\theoremstyle{plain}
\newtheorem{theorem}{Theorem}[section]
\newtheorem{lemma}[theorem]{Lemma}
\newtheorem{corollary}[theorem]{Corollary}

\theoremstyle{definition}
\newtheorem{definition}[theorem]{Definition}
\newtheorem{example}[theorem]{Example}

\title{\textbf{Further constructions of square integer relative Heffter arrays}}

\author{\textbf{Diane Donovan}}
\author{\textbf{Sarah Lawson}\thanks{sarah.lawson@uq.edu.au}}
\author{\textbf{James Lefevre}}
\affil{School of Mathematics and Physics, The University of Queensland, Brisbane, Queensland Australia
ARC Centre of Excellence, Plant Success in Nature and Agriculture, St. Lucia, Queensland, Australia}

\date{}

\begin{document}

\maketitle

\begin{abstract}
A square integer relative Heffter array is an $n \times n$ array whose rows and columns sum to zero, each row and each column has exactly $k$ entries and either $x$ or $-x$ appears in the array for every $x \in \mathbb{Z}_{2nk+t}\setminus J$, where $J$ is a subgroup of size $t$. There are many open problems regarding the existence of these arrays. In this paper we construct two new infinite families of these arrays with the additional property that they are strippable. These constructions complete the existence theory for square integer relative Heffter arrays in the case where $k=3$ and $n$ is prime.  
 
\end{abstract}

\section{Introduction}

Consider $\mathbb{Z}_v=\left\{0, \pm1, \dots, \pm\frac{(v-1)}{2}\right\}$ for some odd integer $v$. A \textit{halfset} $S \subset \mathbb{Z}_v$ is a subset of size $\frac{v-1}{2}$ where, for each $x \in \mathbb{Z}_v \setminus\{0\}$, either $x$ or $-x$ is contained in $S$. A \textit{Heffter system $D(v,k)$}, is a partition of a half set $S$ into blocks of size $k$ such that the block elements sum to $0 \pmod{v}$ \cite{Archdeacon2015}. 
A \textit{Heffter array} denoted $H(m,n;h,k)$, is an $m \times n$ array with the rows forming a \textit{row Heffter system} $D_h=D(2mh+1,h)$, and the columns forming a \textit{column Heffter system} $D_k=D(2nk+1,k)$ with nk=mh. The very first introduction of these objects in \cite{Archdeacon2015} defined these systems as \textit{orthogonal} when each block in $D_h$ intersected with each block in $D_k$ in \textit{exactly} one element thus $h=n$ and $k=m$. However, it is now more commonly accepted that two Heffter systems are orthogonal when each block in $D_h$ intersects with each block in $D_k$ in \textit{at most} one element. This allows for empty cells in the Heffter array. The former definition is now reserved for \textit{tight} Heffter arrays. \par 
A Heffter array therefore meets three key conditions;  each row and column in the array has $h$ and $k$ filled cells respectively, each row and column sums to $0 \pmod{2nk+1}$ and the array entries form a half set of $\mathbb{Z}_{2nk+1}$. 
A \textit{square} Heffter array, $H(n;k)$, occurs when $m=n$ (thus $h=k$) and an \textit{integer} Heffter array is such that the rows and columns sum to zero in $\mathbb{Z}$. In this paper we use the prefix $I$ to denote the integer property.  
 
 \begin{example}    A square integer Heffter array $IH(4;3)$ over $\mathbb{Z}_{25}$.        
    \label{fig:IH_{1}(4;4)}
   \begin{table}[H]
     \centering 
          \begin{tabular}{|p{0.6cm}|p{0.6cm}|p{0.6cm}|p{0.6cm}|}\hline
                 -9&1&8&\\ \hline
                 2&-12&&10 \\ \hline
                 &11&-5&-6 \\ \hline
                 7&&-3&-4 \\ \hline
                 
           \end{tabular}    
    \end{table}
    \centering
   
\end{example}
 
In this current work we consider \textit{relative} Heffter arrays which were introduced in \cite{Costa2020b}. We now state these formally.

\par
\begin{definition} \cite{Costa2020b} 
\label{def:RHA}
Let $v=2nk+t$ be a positive integer, where $t$ divides $2nk$, and let $J$ be the subgroup of $\mathbb{Z}_v$ of order $t$. A $H_t(m,n;h,k)$ Heffter array over $\mathbb{Z}_v$ relative to $J$ is an $m \times n$ partially filled array with elements in $\mathbb{Z}_v$ such that:
\begin{enumerate}
    \item each row contains h filled cells and each column contains k filled cells;
    \item for every $x \in \mathbb{Z}_v \setminus J$, either $x$ or $-x$ appears in the array;
    \item the elements in every row and column sum to 0 in $\mathbb{Z}_v$.
\end{enumerate}
\end{definition}

 The rows (columns) of a relative Heffter array form a \textit{row (column) relative Heffter system} $D_r(v,h) \ (D_c(v,k))$ partitioning a half set of $\mathbb{Z}_v \setminus J $ into blocks of size $h$ ($k$) such that the block elements sum to $0 \pmod{v}$. It is easy to see that setting $t=1$ in Definition \ref{def:RHA} yields the \textit{classical} definition of a Heffter array. \par

The set of absolute values of the entries in a  Heffter array is called the \textit{support} of the array. With slight abuse of notation, the term \textit{support} is also used to denote the set of absolute values of the entries in any given set. 
For $t$ even, the support of a relative Heffter array is the set $\{0, 1, \cdots , nk +\frac{t}{2}\}  \setminus \{0,  (\frac{2nk}{t}+1), 2 (\frac{2nk}{t}+1), \dots, \frac{t}{2}(\frac{2nk}{t}+1)\}$. For $t$ odd, the support is $\{0, 1, \cdots ,  (nk+\frac{t-1}{2})\} \setminus
\{0, (\frac{2nk}{t}+1), 2 (\frac{2nk}{t}+1), \dots, \frac{t-1}{2}(\frac{2nk}{t}+1)\}$. 
 It is useful to note that an $IH_2(m,n;h,k)$ and $IH(m,n;h,k)$ have the same support thus are equivalent. \par

The relative Heffter arrays that we consider in this work are all square and integer and we use the notation $IH_t(n;k)$. 
\begin{example} The following array is an $IH_4(6;3)$ over $\mathbb{Z}_{40}$ relative to $J=\{0, \pm10, 20\}$ with support $\{1, \cdots,9,11,\cdots,19\}$: 
   \begin{table}[H]
     \centering 
     \begin{adjustbox}{width=5cm,center}
          \begin{tabular}{|c|c|c|c|c|c|}\hline
                 -9 & 17 &   &   &   & -8\\ \hline
                 -5 & 1 & 4 &   &   &  \\ \hline
                   & -18 & 11 & 7 &   &  \\ \hline
                   &   & -15 & 12 & 3 &  \\ \hline
                   &   &   & -19 & 13 & 6\\ \hline
                 14 &   &   &   & -16 & 2\\ \hline
           \end{tabular}
           \end{adjustbox}
    \end{table}
    \centering
           \label{fig:IH_{4}(6;3)_v2}
\end{example}

Over the past decade many variations of Heffter arrays have been studied. The interested reader is referred to the 2022 survey \cite{Pasotti2024} for an overview. Recently, \textit{Heffter spaces}, equivalent to an arbitrary number of orthogonal Heffter systems, have been introduced \cite{Buratti2024} and further extended to \textit{relative Heffter spaces} \cite{Johnson2025}. However, there are still many open cases regarding the existence of relative Heffter arrays (or equivalently Heffter spaces with 2 orthogonal Heffter systems). \par
We say parameters $n,k$ are admissible if $3 \leq k\leq n$. The existence of an $IH_t(n;k)$ is completely solved for $t=1,2$ for all admissible parameters \cite{Dinitz2017a, Archdeacon2015a} but there are many unsolved cases for $t >2$. In \cite{Costa2020b}, necessary conditions for the existence of an $IH_t(n;k)$ were presented. 

\begin{theorem}\cite{Costa2020b}
\label{th:IH_t(n;k)_neccesary}
 Suppose that there exists an $IH_t(n;k)$.
 \begin{enumerate}
     \item If t divides $nk$, then $nk \equiv 0 \pmod{4} \text{ or } nk \equiv -t \equiv \pm 1\pmod{4}.$
     \item If $t =2nk$, then k must be even.
     \item If t $\neq  2nk$ does not divide $nk$, then $t+2nk \equiv 0 \pmod{8}.$
 \end{enumerate}
\end{theorem}

Following from this, the authors gave an almost complete proof for the existence of an $IH_k(n;k)$ with the only open case being when $k=5$ and $n\equiv 0 \pmod{4}$. 

 \par Fewer complete results exist for when $t \neq k$. A summary of known results is given in Table \ref{tab:summary_existence_r_sq_int_HA} for admissible parameters. Recall that $t$ must also divide $2nk$. It should be noted that the table does not take into account specific cases of small $t,n,k$ which can be either excluded or constructed through computation.

\begin{figure}[H]
    \label{tab:summary_existence_r_sq_int_HA}
\begin{table}[H]

    \centering
     
\begin{tabular}{|l|l|l|p{5cm}|l|}
\hline
\textbf{$k$} & \textbf{$t$} & \textbf{$n$} & \textbf{Result} & \textbf{Paper} \\
\hline

3 & $n,2n$ & $n$ odd & There exists an $IH_{n}(n;3)$ and an $IH_{2n}(n;3)$ & \cite{Costa2020} \\
\hline

3 & $3n$ & $n$ odd & There is no $IH_{3n}(n;3)$ & \cite{Costa2020b} \\
\hline

3 & 8 & 4 & There is no $IH_8(4;3)$ & \cite{Costa2020b} \\
\hline

$5$ & $5$ & $3 \pmod{4}$ & There exists an $IH_5(n; 5)$ & \cite{Costa2020b} \\
\hline

$5$ & $5$ &  $1, 2 \pmod{4}$ & There is no $IH_5(n; 5)$  & \cite{Costa2020b} \\
\hline

$0 \pmod{4}$ & $t$ & $n$ & There exists an $IH_t(n;k)$ & \cite{Morini2020} \\
\hline

$2 \pmod{4}$ & $t$ & $n$ even & There exists an $IH_t(n;k)$ & \cite{Morini2020} \\
\hline

all $k$ & $1,2$ & all $n$ & There exists an $IH(n;k)$ if and only if $nk \equiv 0,3 \pmod{4}$ & \cite{Dinitz2017a,Archdeacon2015a} \\
\hline

all $k \neq 5$ & $k$ & all $n$ & There exists an $IH_k(n; k)$ if and only if one of the following holds:
\begin{itemize}
    \item $k$ is odd and $n \equiv 0, 3 \pmod{4}$;
    \item $k \equiv 2 \pmod{4}$ and $n$ is even;
    \item $k \equiv 0 \pmod{4}$.
\end{itemize} & \cite{Costa2020b} \\
\hline

\end{tabular}
    \caption{Known existence of an $IH_t(n;k)$}
\end{table}
  
\end{figure}

 Despite being the smallest possible value for $k$, the existence of an $IH_t(n;k)$ for $k=3$ is largely unsolved. It is this case that is the subject of this paper. Corollary \ref{cor:neccesary_for_IH_t(n;3)} provides necessary conditions for the existence of an $IH_t(n;3)$.

\begin{corollary} of Theorem \ref{th:IH_t(n;k)_neccesary} 
\label{cor:neccesary_for_IH_t(n;3)} 

Suppose there exists an $IH_t(n;3)$. 
 \begin{enumerate}
        \item If $t$ divides $3n$, then $n \equiv 0 \pmod{4}$, $n,t \equiv 1 \pmod {4}$ or $n,t \equiv 3 \pmod {4}$.
        \item If $t$ does not divide $3n$, then $t \equiv 2n \pmod{8}$ and  $t \neq 6n$.

\end{enumerate}
\end{corollary}

\begin{proof}
 Assume an $IH_t(n;3)$ exists and that $t$ divides $3n$. From Theorem \ref{th:IH_t(n;k)_neccesary} either $3n \equiv 0 \pmod{4}$ or  $3n \equiv -t \equiv \pm 1 \pmod{4}$ thus either $n \equiv 0 \pmod{4}$ or $n,t \equiv 3 \pmod {4}$ or $n,t \equiv 1 \pmod {4}$ as required.

    Since $k$ is odd then clearly from Theorem \ref{th:IH_t(n;k)_neccesary}, $t \not=6n$. If we assume $t$ does not divide $3n$ then $t+6n \equiv 0 \pmod{8}$ and it follows that $t \equiv 2n \pmod {8}$.
\end{proof}

 In this paper we construct two new infinite families of integer relative Heffter arrays when $k=3$ with $t=6$ and $n \equiv 0,3 \pmod{4}$. These constructions are key to the main theorem of this paper, Theorem \ref{th:IH_t(n;3)_n_prime_exists}.

 \begin{theorem} 
\label{th:IH_t(n;3)_n_prime_exists}
    An $IH_t(n;3)$ with prime $n\geq 3$ exists for
    \begin{enumerate}
        \item  $n\equiv 1 \pmod{4}$ if and only if $t=1,2,n,2n$, and 
        \item  $n\equiv 3 \pmod{4}$ if and only if $t=3,6,n,2n$.
    \end{enumerate}
    
    \label{prop:IHt(n;3) n 1,3 prime existance}
\end{theorem}
  Given an $n \times n$ Heffter array, H, if there exists a transversal of a set with support $\{1, \dots, n\}$ then we say $H$ has a \textit{primary transversal} \cite{Archdeacon2015a}. An array is \textit{shiftable} if it has the same number of positive and negative entries in each row and column \cite{Archdeacon2015a}. If a primary transversal can be removed from a Heffter array leaving  a shiftable array then we say the Heffter array is \textit{strippable} \cite{Archdeacon2015a}. In Section \ref{sec:k=3} we consider the existence of a strippable $IH_t(n;3)$. In Section \ref{sec:t=6} we give new constructions for a strippable $IH_6(n;3)$. We then conclude with the proof of Theorem \ref{th:IH_t(n;3)_n_prime_exists}.

\section{Strippable square integer relative Heffter arrays for \texorpdfstring{$k=3$} .
} 
\label{sec:k=3}
The following lemma tightens the necessary conditions for the existence of an $IH_t(n;k)$ that contains a primary transversal.

\begin{lemma}
\label{lem:n<2nk+t+1}
If an $IH_t(n;k)$ with a primary transversal exists then  $\frac{2nk}{n-1}>t$.
\end{lemma}
\begin{proof}
    Say such an $IH_t(n;k)$ exists then the smallest of the non-zero absolute values of $J$ is  $\frac{2nk}{t}+1$. Thus, for $1, \dots , n$ to exist in the support,  $n<\frac{2nk}{t}+1$ and the result follows. 
\end{proof}

For the remainder of this paper we focus solely on $k=3$.
\begin{lemma}
\label{lem:prim_trans=strippable}
    If an $IH_t(n;3)$ has a primary transversal then it is strippable. 
\end{lemma}
\begin{proof}
    Say an $IH_t(n;3)$ has a primary transversal. Consider the three entries $\{i,r_{1i},r_{2i}\}$ in row $i$. Since $1 \leq|i|\leq n$ and the row sums to zero then $r_{1i}$ and $r_{2i}$ must have opposite signs. An analogous argument follows for the columns hence the array is strippable.
\end{proof}

\begin{lemma}

If an $IH_t(n;3)$ is strippable then $t \leq 6$.
\end{lemma}
\begin{proof}
    Assume a strippable $IH_t(n;3)$ exists, then by definition it has a primary transversal. For all $n\geq 7$, we have $t \leq 6$ from Lemma \ref{lem:n<2nk+t+1}. If $n=4,5,6$, we have $t\leq 7$ however $t = 7$ does not divide $2nk$ when $n = 4,5,6$. Similarly, if $n=3$ then neither $t =7, 8$ divides $2nk$. Thus for all $n$ we have $t \leq 6$.
\end{proof}

We now consider the existence of a strippable $IH_t(n;3)$ for each $t \leq6$. For $t=1,2$, we know that an $IH(n;3)$ exists if and only if $nk \equiv 0,3 \pmod{4}$  \cite{Archdeacon2015a}. The constructs previously provided in \cite{Archdeacon2015a} for these parameters already yield strippable arrays. We also know that an $IH_3(n;3)$ exists if and only if $n \equiv 0,3 \pmod{4}$ \cite{Costa2020b} and again the constructs provided already yield strippable arrays. This leaves the existence of a strippable $IH_t(n;3)$ open when $t=4,5,6$. We now briefly comment on the cases $t=4,5$ then provide detailed discussion and constructions for $t=6$ in Section \ref{sec:t=6}.

Corollary \ref{cor:neccesary_for_IH_t(n;3)} gives the necessary conditions that an $IH_4(n;3)$ exists only if $n\equiv 0,2 \pmod{4}$. If we restrict our search to strippable arrays then we easily see this is not a sufficient condition.  
\begin{lemma}
    There does not exist a strippable $IH_4(4;3)$. 
\end{lemma}
\begin{proof}

    Assume $H$ is a strippable $IH_4(4;3)$ then $H$ has a primary transversal. Using the row and column relative Heffter systems $D_r(28,3)$ and $D_c(28,3)$ respectively, form the systems $D'_i(28,3)=\{\{|x_1|,|x_2|,|x_3| \}  |  \{x_1,x_2,x_3\} \in D_i(28,3)\}$ for $i \in \{r,c\}$. Since $D_r(28,3)$ and $D_c(28,3)$ are orthogonal then $D'_r(28,3)$ and $D'_c(28,3)$ are also orthogonal. Given our assumption, no triple in $D'_i(28,3)$ contains a pair of elements from $\{1, \dots, 4\}$ and the union of the triples in $D'_i(28,3)$ must be the support $\{1, \dots, 14\}\setminus\{7,14\}$. Clearly $D'_r(28,3)$ and $D'_c(28,3)$ must each contain exactly one triple from each of the columns in Table \ref{tab:triples_in_IH443}.  

\begin{table}[ht!]

    \centering
    \begin{tabular}{|c|c|c|c|}\hline
    
        \textbf{1} & \textbf{2} & \textbf{3} & \textbf{4}\\\hline 
         \{1,5,6\}& \{2,6,8\} & \{3,5,8\} & \{4,5,9\}\\ \hline
        \{1,8,9\} & \{2,8,10\} &  \{3,6,9\}& \{4,6,10\}\\ \hline
        \{1,9,10\} & \{2,9,11\} &  \{3,8,11\}& \{4,8,12\} \\ \hline
         \{1,10,11\}&  \{2,10,12\}& \{3,9,12\} & \{4,9,13\}\\ \hline
         \{1,11,12\}& \{2,11,13\}  & \{3,10,13\} &  \\ \hline
         \{1,12,13\}&  & &  \\ \hline
    \end{tabular}
    \label{tab:triples_in_IH443}
    \caption{Possible triples in $D'_r(28,3)$ and $D'_c(28,3)$}
\end{table}
There are only 4 possible combinations of triples such that the union of the triples is $\{1, \dots , 14\}\setminus \{7,14\}$:\begin{align*}
\{\{1,5,6\},\{2,10,12\},\{3,8,11\},\{4,9,13\}\} \\ 
\{\{1,11,12\},\{2,6,8\},\{3,10,13\},\{4,5,9\}\}\\
\{\{1,5,6\},\{2,9,11\},\{3,10,13\},\{4,8,12\}\}\\
\{\{1,12,13\},\{2,9,11\},\{3,5,8\},\{4,6,10\}\}\\
\end{align*}
By inspection, no two sets are orthogonal thus $H$ cannot exist.
\end{proof}
If we remove the primary transversal condition an $IH_4(4;3)$ is simple to generate.
\begin{example}
The following array is an $IH_{4}(4;3)$ over $\mathbb{Z}_{28}$ without a primary transversal: 
           \label{fig:IH_{4}(4;3)}
\begin{figure}[H]
   \begin{table}[H]
     \centering 
          \begin{tabular}{|p{0.55cm}|p{0.55cm}|p{0.55cm}|p{0.55cm}|}\hline
                 10 & 1 &   & -11\\ \hline
                 -13 & 5 & 8 &  \\ \hline
                   & -6 & 4 & 2\\ \hline
                 3 &   & -12 & 9\\ \hline
           \end{tabular}
    \end{table}    
\end{figure}
\centering

\end{example}
Computations show that a strippable $IH_4(n;3)$ does exist for small even $n>4$ but it is yet unknown if the necessary conditions are sufficient for all even $n\not=4$.

If $t=5$, we can see that an $IH_5(n;3)$ exists only if $n\equiv 0,5 \pmod{20}$ by Corollary \ref{cor:neccesary_for_IH_t(n;3)}. While we cannot yet confirm if this is sufficient, a strippable $IH_5(5;3)$ does exist for the smallest possible $n=5$, unlike the case for $t=4$. The construction for an $IH_n(n;3)$ given in \cite{Costa2020} provides such an array.

\section{The case for \texorpdfstring{$t=6$}.}
\label{sec:t=6}
We now give two constructions for a strippable $IH_6(n;3)$ for $n \equiv 0,3 \pmod{4}$.  

For $ i \in \{1, \dots, n\},$ the $i^{th}$ diagonal, $D_i$, of a square array is defined as the set of cells $\{(i,1),(i+1,2), \dots , (i+n-1,n)\}$ with arithmetic performed modulo $n$. An $n \times n$ array is \textit{cyclically k-diagonal} if its nonempty cells comprise of exactly $k$ consecutive diagonals, $D_i,D_{i+1}, \dots ,D_{i+k-1}$ for some $i \in \{1, \dots, n\}$. Again, arithmetic is performed modulo $n$. The arrays constructed in Theorem \ref{th:IH_6(n;3) _constructions} are cyclically 3-diagonal as their nonempty cells are exactly $D_2, D_1,D_n$. We present the constructions using the standard notation $diag(r,c,s, \Delta_1, \Delta_2,l)$ \cite{Dinitz2017a} to fill in the cells of an array, $H$, where $H_{r+i\Delta_1,c+i\Delta_1}=s+i \Delta_2 \text{ for } i \in \{0, \cdots, l-1\}$. Row and columns indices are modulo $n$. The constructed array has the property that  $|H_{i,i+1}|+|H_{i+1,i}|=4n+4$, this is discussed further following Theorem \ref{th:IH_6(n;3) _constructions}. To highlight this property we use the notation $Z,Z'$ to identify the sequence pairs where $|z_i|+|z'_i|=4n+4$.

\begin{theorem}
\label{th:IH_6(n;3) _constructions}
A strippable $IH_6(n;3)$ exists if $n \equiv 0,3 \pmod{4}$. 
\end{theorem}

\begin{proof}

\textbf{Case $n \equiv 3 \pmod{4}$}\par
Consider the $n \times n$ array $H$ with $n \equiv 3 \pmod{4}$ constructed by the following:\par
\onehalfspacing{
$A:diag(1,1,-(n-1),1,1,\frac{n-1}{2})$\par
$B:diag(\frac{n+3}{2}, \frac{n+3}{2}, \frac{n-3}{2},1,-1,\frac{n-3}{2})$\par
$C:diag(1,n,-(n+2),2,-1,\frac{n+1}{4})$\par
$C':diag(n,1,3n+2,2,-1,\frac{n+1}{4})$\par
$D:diag(1,2,2n+1,2,-1,\frac{n+1}{4})$\par
$D':diag(2,1,-(2n+3),2,-1,\frac{n+1}{4})$\par
$E:diag(\frac{n+3}{2},\frac{n+1}{2},-\frac{11n+7}{4},2,1,\frac{n+1}{4})$\par
$E':diag(\frac{n+1}{2}, \frac{n+3}{2},\frac{5n+9}{4}, 2,1, \frac{n+1}{4})$\par
$F:diag(\frac{n+3}{2},\frac{n+5}{2},\frac{9n+13}{4},2,1,\frac{n-3}{4})$\par
$F':diag(\frac{n+5}{2}, \frac{n+3}{2}, -\frac{7n+3}{4},2,1, \frac{n-3}{4})$ \par}

Plus the adhoc values: \par
$H_{\frac{n+1}{2}, \frac{n+1}{2}}= n$, 
$H_{n,n}=-\frac{n-1}{2}$\\ \par

 An $IH_6(11,3)$ built using this construction is given in Example \ref{ex:IH_{6}(11;3)} alongside Figure \ref{fig:n=11_sequences} illustrating
 the general pattern.\\ \par 

\begin{minipage}{0.5 \textwidth}

\begin{example} An $IH_{6}(11;3)$ over $\mathbb{Z}_{72}$
   \begin{table}[H]
     \centering 
        \begin{adjustbox}{height=0.14\textheight}
      \begin{tabular}{|wc{0.35cm}|wc{0.35cm}|wc{0.35cm}|wc{0.35cm}|wc{0.35cm}|wc{0.35cm}|wc{0.35cm}|wc{0.35cm}|wc{0.35cm}|wc{0.35cm}|wc{0.35cm}|}\hline
                 -10 & 23 &   &   &   &   &   &   &   &   & -13 \\[1pt] \hline
                 -25 & -9 & 34 &   &   &   &   &   &   &   &  \\[1pt] \hline
                   & -14 & -8 & 22 &   &   &   &   &   &   &  \\ [1pt]\hline
                   &   & -26 & -7 & 33 &   &   &   &   &   &  \\[1pt] \hline
                   &   &   & -15 & -6 & 21 &   &   &   &   &  \\[1pt] \hline
                   &   &   &   & -27& 11 & 16 &   &   &   &  \\[1pt] \hline
                   &   &   &   &   & -32 & 4 & 28 &   &   &  \\ [1pt]\hline
                   &   &   &   &   &   & -20 & 3 & 17 &   &  \\ [1pt]\hline
                   &   &   &   &   &   &   & -31 & 2 & 29 &  \\ [1pt]\hline
                   &   &   &   &   &   &   &   & -19 & 1 & 18\\ [1pt]\hline
                 35 &   &   &   &   &   &   &   &   & -30 & -5\\ [1pt]\hline
           \end{tabular}
           \end{adjustbox}
    \end{table}
    \centering
    
           \label{ex:IH_{6}(11;3)}
\end{example}
\end{minipage}
\begin{minipage}{0.5\textwidth}

      \begin{figure}[H]
      \centering 
     \begin{adjustbox}
     {height=0.14\textheight}
      \begin{tabular}{|wc{0.35cm}|wc{0.35cm}|wc{0.35cm}|wc{0.35cm}|wc{0.35cm}|wc{0.35cm}|wc{0.35cm}|wc{0.35cm}|wc{0.35cm}|wc{0.35cm}|wc{0.35cm}|}\hline
          $A$&$D$&&&&&&&&&$C$\\ \hline
          $D'$&$A$&$C'$&&&&&&&&\\ \hline
          &$C$&$A$&$D$&&&&&&&\\ \hline
          &&$D'$&$A$&$C'$&&&&&&\\ \hline
          &&&$C$&$A$&$D$&&&&&\\ \hline
          &&&&$D'$&$n$&$E'$&&&&\\ \hline
          &&&&&$E$&$B$&$F$&&&\\ \hline
          &&&&&&$F'$&$B$&$E'$&&\\ \hline
          &&&&&&&$E$&$B$&$F$&\\ \hline
          &&&&&&&&$F'$&$B$&$E'$\\ \hline
          $C'$&&&&&&&&&$E$&\tiny{$-\frac{n-1}{2}$} \\ \hline
          \end{tabular}  
          \end{adjustbox}
    \caption{General construction of an $IH_t(n;3)$ for $n=11$}
     \label{fig:n=11_sequences}
          \end{figure}
          \vspace{-1.6cm}
\end{minipage}

 \vspace{1cm}
 We now prove this construction meets the required conditions for an integer relative Heffter array. We consider the row and column sums separated into cases. 
 In the following tables, $Z_i$ denotes the $i_{th}$ term in sequence $Z$. The entries in each row are given in Table \ref{tab:IH6(n3;3)_row_sum} and the column entries are given in Table \ref{tab:IH6(n;3)_column_sum}.  It is easily seen that the entries in each row and column sum to zero. 
 \begin{figure}[H]
    \begin{table}[H]
        \centering 
        \begin{adjustbox}{width=1\textwidth}
          \begin{tabular}{|l|l|l|l|}\hline
          \textbf{Rows} &\multicolumn{3}{c|}{\textbf{Row entries}}\\ \hline
          $odd \ r \in [1, \frac{n-1}{2}]$&$A_r= -(n-1)+r-1$&$D_{\frac{r+1}{2}} = (2n+1)-\frac{r-1}{2}$&$C_{\frac{r+1}{2}}= -(n+2)-\frac{r-1}{2}$ \\[5pt]
          \hline
          $even \ r \in [1, \frac{n-1}{2}]$&$D'_{\frac{r}{2}}=-(2n+3)-\frac{r-2}{2}$&$A_r=-(n-1)+r-1$&$C'_{\frac{r+2}{2}}=(3n+2)-(\frac{r-2}{2}+1)$\\ [5pt]\hline
           $\frac{n+1}{2}$&$D_{\frac{n+1}{4}}'=-(2n+3)-\left(\frac{n+1}{4}-1\right)$&$n$&$E'_1= \frac{5n+9}{4}$ \\[5pt] \hline
           
           $odd \ r \in [\frac{n+3}{2}$,$n-1]$& $E_{\frac{2r-n+1}{4}}=-\frac{11n+7}{4}+\frac{2r-n-3}{4}$&$B_\frac{2r-n-1}{2}=\frac{n-3}{2}-\frac{2r-n-3}{2}$&$F_{\frac{2r-n+1}{4}}=\frac{9n+13}{4}+\frac{2r-n-3}{4}$ \\ [5pt]\hline 
           $even \ r \in [\frac{n+3}{2}$,$n-1]$&$F'_{\frac{2r-n-1}{4}}=-\frac{7n+3}{4}+\frac{2r-n-5}{4}$ &$B_{\frac{2r-n-1}{2}}=\frac{n-3}{2}-(\frac{2r-n-5}{2}+1)$&$E'_{\frac{2r-n+3}{4}}=\frac{5n+9}{4}+(\frac{2r-n-5}{4}+1)$ \\ [5pt]\hline
           $n$&$C_1'=3n+2$&$E_{\frac{n+1}{4}}=-\left(\frac{11n+7}{4}\right)+\left(\frac{n+1}{4}-1\right)$&$\ -\frac{n-1}{2}$ \\[5pt] \hline
          \end{tabular}
          \end{adjustbox}
              \caption{Row entries for an $IH_6(n;3)$ for $n \equiv 3 \pmod{4}$}
              \label{tab:IH6(n3;3)_row_sum}
    \end{table}
\end{figure}

    %%%%COLUMNS
    \begin{figure}[H]
    \begin{table}[H]
        \centering 
        \begin{adjustbox}{width=1\textwidth}
          \begin{tabular}{|l|l|l|l|}\hline
          \textbf{Columns} &\multicolumn{3}{c|}{\textbf{Column entries}}\\ \hline
          $odd \ c \in [1, \frac{n-1}{2}]$&$A_c= -(n-1)+c-1$&$D'_{\frac{c+1}{2}} = -(2n+3)-\frac{c-1}{2}$&$C'_{\frac{c+1}{2}}= 3n+2-\frac{c-1}{2}$\\[2pt]
          \hline
          $even \ c \in [1, \frac{n-1}{2}]$&$D_{\frac{c}{2}}=(2n+1)-\frac{c-2}{2}$&$A_c=-(n-1)+c-1$&$C_{\frac{c+2}{2}}=-(n+2)-1-\frac{c-2}{2}$\\ [5pt]\hline
           $\frac{n+1}{2}$&$D_{\frac{n+1}{4}}=2n+1-\left(\frac{n+1}{4}-1\right)$&$n$&$E_1= -\frac{11n+7}{4}$ \\ [5pt]\hline
           
           $odd \ c \in [\frac{n+3}{2}$,$n-1]$& $E'_{\frac{2c-n+1}{4}}=\frac{5n+9}{4}+\frac{2c-n-3}{4}$&$B_\frac{2c-n-1}{2}=\frac{n-3}{2}-\frac{2c-n-3}{2}$&$F'_{\frac{2c-n+1}{4}}=-\frac{7n+3}{4}+\frac{2c-n-3}{4}$ \\ [5pt]\hline 
           
           $even \ c \in [\frac{n+3}{2}$,$n-1]$& $E_{\frac{2c-n+3}{4}}=-\frac{11n+7}{4}+1+\frac{2c-n-5}{4}$&$B_{\frac{2c-n-1}{2}}=\frac{n-3}{2}-1-\frac{2c-n-5}{2}$&$F_{\frac{2c-n-1}{4}}=\frac{9n+13}{4}+\frac{2c-n-5}{4}$\\ [5pt]\hline
           $n$&$C_1=-(n+2)$&$E'_{\frac{n+1}{4}}=\left(\frac{5n+9}{4}\right)+\left(\frac{n+1}{4}-1\right)$&$\ -\frac{n-1}{2}$ \\[5pt] \hline
          \end{tabular}
          \end{adjustbox}
              \caption{Column entries for an $IH_6(n;3)$ for $n \equiv 3 \pmod{4}$}
                \label{tab:IH6(n;3)_column_sum}
    \end{table}
\end{figure}

We now show that the construction gives the required support. We consider the union of the sets containing the absolute values for each sequence and the adhoc entries.

\begin{align*}
supp(H)
&=\\
\biggl\{
    &\underset{supp(B)}{\left[1,\frac{n-3}{2}\right]}, 
    \underset{H_{n,n}}{\frac{n-1}{2}},
    \underset{supp(A)}{\left[\frac{n+1}{2}, n-1\right]}, 
    \underset{H_{\frac{n+1}{2}, \frac{n+1}{2}}}{n},
    \underset{supp(C)}{\left[n+2, \frac{5n+5}{4} \right]},
    \underset{supp(E')}{\left[\frac{5n+9}{4},\frac{3n+3}{2}\right]},
    \\
    &\underset{supp(F')}{\left[\frac{3n+5}{2},\frac{7n+3}{4}\right]},
    \underset{supp(D)}{\left[\frac{7n+7}{4},2n+1\right]},
    \underset{supp(D')}{\left[2n+3,\frac{9n+9}{4} \right]},
    \underset{supp(F)}{\left[\frac{9n+13}{4},\frac{5n+3}{2} \right]},\\
    &\underset{supp(E)}{\left[\frac{5n+5}{2}, \frac{11n+7}{4}\right]},
    \underset{supp(C')}{\left[\frac{11n+11}{4}, 3n+2 \right]}
    \biggr\}\\
    &=[1,3n+2]\setminus\{n+1,2n+2\} \text{ as required} 
\end{align*}

 It is clear that the construction places three non-empty cells into each row and column thus $H$ is an $IH_6(n;3)$ for $n \equiv 0 \pmod {4}$. Furthermore, the main diagonal has support $\{1, \dots, n\}$ and is therefore a primary transversal. Since $k=3$, it follows from Lemma \ref{lem:prim_trans=strippable} that the Heffter array is strippable.\\

\textbf{Case for $n \equiv 0 \pmod{4}$}\\
Consider the $n \times n$ array $H$ with $n \equiv 0 \pmod{4}$ constructed by the following:\par
\onehalfspacing{
$A:diag(1,1,-(n-2),1,1,\frac{n-4}{2})$\par
$B:diag(\frac{n}{2},\frac{n}{2},\frac{n-2}{2},1,-1, \frac{n-4}{2})$\par
$C:diag((1,n,-(n+3),2,-1,\frac{n}{4})$\par
$C':diag(n,1,3n+1,2,-1,\frac{n}{4})$\par
$D:diag(1,2,2n+1,2,-1,\frac{n-4}{4})$\par
$D':diag(2,1,-(2n+3),2,-1,\frac{n-4}{4})$\par
$E:diag(\frac{n-2}{2},\frac{n}{2},\frac{9n+8}{4},2,1,\frac{n}{4})$\par
$E':diag(\frac{n}{2}, \frac{n-2}{2}, -\frac{7n+8}{4},2,1,\frac{n}{4})$\par
$F:diag(\frac{n+2}{2},\frac{n}{2},-\frac{11n+4}{4},2,1,\frac{n-4}{4})$\par
$F':diag(\frac{n}{2},\frac{n+2}{2}, \frac{5n+12}{4},2,1,\frac{n-4}{4})$}\par
\par Plus adhoc cells:\par
$H_{\frac{n-2} {2},\frac{n-2}{2}}=-n$, 
$H_{n-2,n-2}=-(n-1)$, 
$H_{n-2,n-1}=\frac{5n+4}{2}$, 
$H_{n-1,n-2}=-\frac{3n+4}{2}$, 
$H_{n-1,n-1}=\frac{n}{2}$,  
$H_{n-1,n}=n+2$,  
$H_{n,n-1}=-(3n+2)$,  
$H_{n,n}=1$.\par

\singlespacing
An $IH_6(12,3)$ built using this construction is given in Example \ref{ex:IH_{6}(12;3)}. Figure \ref{fig:n=12_sequences} illustrates the general pattern.

         \begin{example}
         An $IH_{6}(12;3)$ over $\mathbb{Z}_{79}$
          \begin{table}[H]
        \centering 
     \begin{adjustbox}{height=0.17\textheight}

     \begin{tabular}{|wc{0.35cm}|wc{0.35cm}|wc{0.35cm}|wc{0.35cm}|wc{0.35cm}|wc{0.35cm}|wc{0.35cm}|wc{0.35cm}|wc{0.35cm}|wc{0.35cm}|wc{0.35cm}|wc{0.35cm}|} \hline
   -10&25 &&&&&&&&&&-15\\ \hline
            -27&-9&36&&&&&&&&&\\ \hline
            &-16&-8&24&&&&&&&&\\ \hline
            &&-28&-7&35&&&&&&&\\ \hline
            &&&-17&-12&29&&&&&&\\ \hline
            &&&&-23&5&18&&&&&\\ \hline
            &&&&&-34&4&30&&&&\\ \hline
            &&&&&&-22&3&19&&&\\ \hline
            &&&&&&&-33&2&31&&\\ \hline
            &&&&&&&&-21&-11&32&\\ \hline
            &&&&&&&&&-20&6&14\\ \hline
            37&&&&&&&&&&-38&1\\ \hline
           \end{tabular}
           \end{adjustbox}
           
           \label{ex:IH_{6}(12;3)}      
 \end{table}
  \end{example}

\begin{figure}[H]

     \centering 
     \begin{adjustbox}
       {height=0.17\textheight}

     \begin{tabular}%
     {|wc{0.70cm}|wc{0.70cm}|wc{0.70cm}|wc{0.70cm}|wc{0.70cm}|wc{0.70cm}|wc{0.70cm}|wc{0.70cm}|wc{0.70cm}|wc{0.70cm}|wc{0.70cm}|wc{0.70cm}|} \hline
          $A$&$D$&&&&&&&&&&$C$ \\ \hline
          $D'$&$A$&$C'$&&&&&&&&&\\ \hline
          &$C$&$A$&$D$&&&&&&&&\\ \hline
          &&$D'$&$A$&$C'$&&&&&&&\\\hline
          &&&$C$&$-n$&$E$&&&&&&\\ \hline
          &&&&$E'$&$B$&$F'$&&&&&\\ \hline
          &&&&&$F$&$B$&$E$&&&&\\ \hline
          &&&&&&$E'$&$B$&$F'$&&&\\ \hline
          &&&&&&&$F$&$B$&$E$&& \\ \hline

          &&&&&&&&$E'$&\tiny$-(n-1)$&$\frac{5n+4}{2}$&  \\[3pt]  \hline
          &&&&&&&&&\tiny$-\frac{3n+4}{2}$&$\frac{n}{2}$&\tiny$n+2$  \\  \hline
          $C'$&&&&&&&&&&\tiny-$(3n+2)$&1  \\ \hline
          \end{tabular}  
           \end{adjustbox}
    \caption{General construction of an $IH_t(n;k)$ for $n=12$}
     \label{fig:n=12_sequences}
    \end{figure}

 Following the same approach as for $n \equiv 3 \pmod{4}$ we consider the row and column sums separated into cases. 
 The entries in each row are given in Table \ref{tab:IH6(n0;3)_row_sum} and the entries in each column are given in Table \ref{tab:IH6(n0;3)_column_sum}.  It is easily seen again that the entries in each row and column sum to zero. 

 \begin{figure}[H]
    \begin{table}[H]
        \centering 
        \begin{adjustbox}{width=1\textwidth}
          \begin{tabular}{|l|l|l|l|}\hline
          \textbf{Rows} &\multicolumn{3}{c|}{\textbf{Row entries}}\\ \hline
          $odd \ r \in [1, \frac{n-4}{2}]$&$A_r= -(n-2)+r-1$&$D_{\frac{r+1}{2}} = 2n+1-\frac{r-1}{2}$&$C_{\frac{r+1}{2}}= -(n+3)-\frac{r-1}{2}$\\[5pt]
          \hline
          $even \ r \in [1, \frac{n-4}{2}]$&$D'_{\frac{r}{2}}=-(2n+3)-\frac{r-2}{2}$&$A_r=-(n-2)+r-1$&$C'_{\frac{r+2}{2}}=3n+1-1-\frac{r-2}{2}$\\ \hline
           $\frac{n-2}{2}$&$C_{\frac{n}{4}}=-(n+3)-\left(\frac{n}{4}-1\right)$&$-n$&$E_1= \frac{9n+8}{4}$ \\ [5pt]\hline
           
           $odd \ r \in [\frac{n}{2},n-3]$& $F_{\frac{2r-n+2}{4}}=-\frac{11n+4}{4}+\frac{2r-n-2}{4}$&$B_\frac{2r-n+2}{2}=\frac{n-2}{2}-1-\frac{2r-n-2}{2}$&$E_{\frac{2r-n+6}{4}}=\frac{9n+8}{4}+1+\frac{2r-n-2}{4}$ \\ [5pt]\hline 
           $even \ r \in [\frac{n}{2},n-3]$& $E'_{\frac{2r-n+4}{4}}=-\frac{7n+8}{4}+\frac{2r-n}{4}$&$B_{\frac{2r-n+2}{2}}=\frac{n-2}{2}-\frac{2r-n}{2}$&$F'_{\frac{2r-n+4}{4}}=\frac{5n+12}{4}+\frac{2r-n}{4}$\\ [5pt]\hline        
           $n-2$&$E_{\frac{n}{4}}=-\frac{7n+8}{4}+(\frac{n}{4}-1)$&$-(n-1)$&$\frac{5n+4}{2}$\\ [5pt]\hline
           $n-1$&$-\frac{3n+4}{2}$&$\frac{n}{2}$&$n+2$\\ [5pt]\hline
           $n$&$C_1'=3n+1$&$\ -(3n+2)$&$\ 1$ \\ [5pt]\hline
          \end{tabular}
          \end{adjustbox}
              \caption{Row entries for an $IH_6(n;3)$ for $n \equiv 0 \pmod{4}$}
              \label{tab:IH6(n0;3)_row_sum}
    \end{table}
\end{figure}

 \begin{figure}[H]
    \begin{table}[H]
        \centering 
        \begin{adjustbox}{width=1\textwidth}
          \begin{tabular}{|l|l|l|l|}\hline
          \textbf{Columns} &\multicolumn{3}{c|}{\textbf{Column entries}}\\ \hline
          $odd \ c \in [1, \frac{n-4}{2}]$&$A_c= -(n-2)+c-1$&$D'_{\frac{c+1}{2}} = -(2n+3)-\frac{c-1}{2}$&$C'_{\frac{c+1}{2}}= 3n+1-\frac{c-1}{2}$\\ [5pt]
          \hline
          $even \ c \in [1, \frac{n-4}{2}]$&$D_{\frac{c}{2}}=(2n+1)-\frac{c-2}{2}$&$A_c=-(n-2)+c-1$&$C_{\frac{c+2}{2}}=-(n+3)-1-\frac{c-2}{2}$\\ [5pt]\hline
           $\frac{n-2}{2}$&$C'_{\frac{n}{4}}=3n+1-\left(\frac{n}{4}-1\right)$&$-n$&$E'_1= -\frac{7n+8}{4}$ \\ [5pt]\hline
           
           $odd \ c \in [\frac{n}{2}$,$n-3]$& $F'_{\frac{2c-n+2}{4}}=\frac{5n+12}{4}+\frac{2c-n-2}{4}$&$B_\frac{2c-n+2}{2}=\frac{n-2}{2}-1-\frac{2c-n-2}{2}$&$E'_{\frac{2c-n+6}{4}}=-\frac{7n+8}{4}+1+\frac{2c-n-2}{4}$ \\ [5pt]\hline 
           
           $even \ c \in [\frac{n}{2},n-3]$& $E_{\frac{2c-n+4}{4}}=\frac{9n+8}{4}+\frac{2c-n}{4}$&$B_{\frac{2c-n+2}{2}}=\frac{n-2}{2}-\frac{2c-n}{2}$&$F_{\frac{2c-n+4}{4}}=-\frac{11n+4}{4}+\frac{2c-n}{4}$\\[5pt] \hline
           $n-2$&$E_{\frac{n}{4}}=\frac{9n+8}{4}+\left(\frac{n}{4}-1 \right)$&$-(n-1)$&$-\frac{3n+4}{2}$\\ \hline
           $n-1$&$\frac{5n+4}{2}$&$\frac{n}{2}$&$-(3n+2)$\\ [5pt]\hline
           $n$&$C_1=-(n+3)$&$n+2$&$1$ \\ [5pt]\hline
          \end{tabular}
          \end{adjustbox}
              \caption{Column entries for an $IH_6(n;3)$ for $n \equiv 0 \pmod{4}$}
              \label{tab:IH6(n0;3)_column_sum}
    \end{table}
\end{figure}

By considering the union of the sets containing the absolute values of each sequence and the adhoc entries we show that the array gives the required support.

\begin{align*}
supp(H)
&=\\
\biggl\{
    &\underset{H_{n,n}}{1}, 
    \underset{supp(B)}{\left[2, \frac{n-2}{2}\right]}, 
    \underset{H_{n-1,n-1}}{\frac{n}{2}},
    \underset{supp(A)}{\left[\frac{n+2}{2},n-2\right]},
    \underset{H_{n-2,n-2}}{n-1},
    \underset{H_{\frac{n-2}{2},\frac{n-2}{2}}}{n},
    \underset{H_{n-1,n}}{n+2},\underset{supp(C)}{\left[n+3,\frac{5n+8}{4} \right]},
    \\
    &\underset{supp(F')}{\left[\frac{5n+12}{4}, \frac{3n+2}{2}\right]},
    \underset{H_{n-1,n-2}}{\frac{3n+4}{2}},
    \underset{supp(E')}{\left[\frac{3n+6}{2}, \frac{7n+8}{4}\right]},
    \underset{supp(D)}{\left[\frac{7n+12}{4}, 2n+1\right]},
    \underset{supp(D')}{\left[2n+3, \frac{9n+4}{4}\right]},\\
    &
    \underset{supp(E)}{\left[\frac{9n+8}{4}, \frac{5n+2}{2}\right]},
   \underset{H_{n-2,n-1}}{\frac{5n+4}{2}},
    \underset{supp(F)}{\left[\frac{5n+6}{2}, \frac{11n+4}{4}\right]},
    \underset{supp(C')}{\left[\frac{11n+8}{4}, 3n+1\right]},
    \underset{H_{n,n-1}}{3n+2},
    \biggr\}\\
    &=[1,3n+2]\setminus\{n+1,2n+2\} \text{ as required.}
\end{align*}

Clearly the construction places three non-empty cells into each row and column thus $H$ is an $IH_6(n;3)$ for $n \equiv 0 \pmod {4}$. Furthermore, the main diagonal has support $\{1, \dots, n\}$ and is therefore a primary transversal. Since $k=3$, it follows from Lemma \ref{lem:prim_trans=strippable} that the Heffter array is strippable.

\end{proof}

In \cite{Archdeacon2015a} the authors used existing integer current assignments introduced in \cite{Youngs1970} to construct an $IH(n;3)$ for $n \equiv 0 \pmod{4}$ and $n \equiv 1 \pmod{4}$. M\"obius and cylindrical current ladder graphs were used with the set of assigned currents forming $\mathbb{Z}_{2nk+1} \setminus \{0\}$. Motivated by this previous work, the constructions presented in the proof of Theorem \ref{th:IH_6(n;3) _constructions} were developed by creating current assignments for ladder graphs with $n \equiv 0,3 \pmod{4}$ rungs with the set of assigned currents being $\mathbb{Z}_{6n+6} \setminus \{n+1,2n+2,3n+3\}$. It is worth noting that the resulting Heffter arrays have a negative lower diagonal and a positive upper diagonal as well as the main diagonal having support $\{1, \dots,n\}$. Along with the property that $|H_{i,i+1}|+|H_{i+1,i}|=4n+4$, these properties are a direct result of the underlying current ladder graph's properties. In particular, the graph is constructed such that the set of currents assigned to the rungs is $\{1, \dots, n\}$. These values then map to the main diagonal of the resulting Heffter array. For a more detailed discussion regarding the construction of current graphs the reader is referred to the original paper \cite{Youngs1970} and to the useful constructions provided in \cite{Anderson1982}.\par

\begin{theorem}
\label{th:IH_6(n;3)_iff_n=0,3}
    A (strippable) $IH_6(n;3)$ exists if and only if $n \equiv 0,3 \pmod{4}$. 
 \end{theorem}
 \begin{proof}
   Corollary \ref{cor:neccesary_for_IH_t(n;3)} provides the necessary condition that an $IH_6(n;3)$ exists only if $n \equiv 0,3 \pmod{4}$. Theorem \ref{th:IH_6(n;3) _constructions} shows that these conditions are also sufficient.

 \end{proof}
As a consequence of Theorem \ref{th:IH_6(n;3)_iff_n=0,3} we can now prove Theorem \ref{th:IH_t(n;3)_n_prime_exists}.

\begin{proof} \textit{(of Theorem \ref{th:IH_t(n;3)_n_prime_exists}) }
\par By definition, $t$ must divide $2nk=6n$. Since $t \not= 6n$ by Corollary \ref{cor:neccesary_for_IH_t(n;3)} and $n$ is prime, $t$ must necessarily belong to $\{1,2,3,6,n,2n,3n\}$. 
\begin{itemize}
    \item $n\equiv 1 \pmod{4}$\\
    As summarised in Table \ref{tab:summary_existence_r_sq_int_HA}, existing results prove there is no $IH_t(n;3)$ for $t=3,3n$ \cite{Costa2020b} and Theorem \ref{th:IH_6(n;3)_iff_n=0,3} shows an $IH_6(n;3)$ cannot exist. Thus necessarily $t \in \{1,2,n,2n\}$.\par These conditions are sufficient for $n \equiv 1  \pmod{4}$, with the construction of an $IH_t(n;3)$ when $t=1,2$  being given in \cite{Archdeacon2015a} and for $t=n,2n$ given in \cite{Costa2020} as stated in Table  \ref{tab:summary_existence_r_sq_int_HA}.\\

\item $n\equiv 3 \pmod{4}$\\
    Referring to Table \ref{tab:summary_existence_r_sq_int_HA} we can see an $IH_t(n;3)$  does not exist for $t=1,2$ \cite{Archdeacon2015a} nor for $t=3n$\cite{Costa2020b}. Thus it is necessary that $t\in \{3,6,n,2n\}$.\par
    The existence of an $IH_t(n;3)$ for $t=3$ is given in \cite{Costa2020b} and for $t=n,2n$ in \cite{Costa2020}. To complete the proof, Theorem \ref{th:IH_6(n;3) _constructions} proves the existence of an $IH_t(n;3)$.
\end{itemize}
\end{proof}

 \bibliographystyle{abbrv}
\bibliography{./bibliography.bib}

\end{document}